\documentclass[a4paper,10pt]{article}
\usepackage[utf8]{inputenc}
\usepackage{amsmath,amsfonts,amsthm,amssymb}
\usepackage{breqn,fancyref}
\usepackage[hidelinks]{hyperref}
\usepackage{enumerate}
\allowdisplaybreaks

\newcommand{\aumlaut}[0]{\"a}
\newcommand{\zespolone}[0]{\mathbb{C}}
\newcommand{\rzeczywiste}[0]{\mathbb{R}}

\newcommand{\dom}{\text{dom}}
\newcommand{\email}{\href{mailto:szymon.myga@im.uj.edu.pl}{\nolinkurl{szymon.myga@im.uj.edu.pl}}}
\newcommand{\Monge}{Monge-Amp\`{e}re }

\newtheorem{Thm}{Theorem}[section]
\newtheorem{Prop}[Thm]{Proposition}
\newtheorem{Lem}[Thm]{Lemma}
\newtheorem{Cor}[Thm]{Corollary}
\newtheorem*{ThmA}{Theorem}
\newtheorem*{LemA}{Lemma}

\theoremstyle{definition}
\newtheorem*{Def}{Definition}
\newtheorem*{Rem}{Remark}
\newtheorem*{Ex}{Example}

\title{Optimal transport on completely integrable toric manifolds}
\author{Szymon Myga}
\date{}

\begin{document}

\maketitle



\begin{abstract}
We show that existence and uniqueness of solutions to transported \Monge problem on complex compact toric manifold follows easily from the real theory of optimal transportation. 
\end{abstract}

\section*{Introduction}

Let $(X,\omega)$ be a compact K\aumlaut hler manifold of real dimension $2n$, i.e. $X$ is a complex manifold and one can find a hermitian metric on it whose fundemental form $\omega$ is closed, thus making $(X,\omega)$ into a symplectic manifold. We assume that $X$ is toric - there is a real torus $T^k$ acting on it by automorphisms of $\omega$. Such an action also generates a Lie algebra homomorphism from the Lie algebra of the torus $\mathfrak{t} \simeq \rzeczywiste^k$ into the Lie algebra of the vector fields of $X$. This action can be extended to the holomorphic action of complexified torus $T_c^k \simeq (\zespolone^*)^k$. We also assume that the action is completely integrable and effective. That is, we want the torus to be of greatest possible dimension ($k=n$) and we want the trivial automorphism to only come from the identity element. Lastly, we want the action to be Hamiltonian, so we assume that there is a moment map: an action invariant function $m:X \rightarrow (\rzeczywiste^n)^*$, with $(\rzeczywiste^n)^*$ being the Lie algebra dual to $\mathfrak{t}$, such that for every element $t \in \rzeczywiste^n$
\[
 -d\langle m(p),t\rangle = \omega_p(t^{\#},\cdot) 
\]
with $t^{\#}$ being the vector field generated by $t$ and $\langle m(p),t\rangle$ being the value of the linear form $m(p)$ at $t$.

In this setting one can prove that the image of $X$ through $m$ is a compact convex polytope in $\rzeczywiste^n$ with non-empty interior. Moreover this image does not depend on the choice of particular $\omega$ in an invariant cohomology class. 

Suppose a probability measure with density $1/C < g(p) < C$ is given on a moment polytope $P$ for some toric K\aumlaut hler manifold $(X,\omega)$ with completely integrable torus action. Following the preprint~\cite{BermanWitt} one can define a notion of complex transported \Monge measure $MA_g$ on $X$ which corresponds to the \Monge measure that appears in the theory of optimal transportation of measures. Then the natural question to ask is whether the equation
\[
 MA_g(\phi) = \mu
\]
has a unique (up to an additive constant) solution for any invariant measure $\mu$ that does not put any mass on polar sets, i.e. the sets that are $-\infty$ loci of plurisubharmonic functions. This is a technical assumption that comes up when one tries to define a \Monge operator for singualr functions. A partial answer to the above question is provided in~\cite{BermanWitt} where the authors prove the existence of solutions and uniqueness for a subclass of measures.

In the setting sketched above we prove the following theorem:

\begin{ThmA}
For any invariant probability measure $\mu$ that does not put mass on polar sets there is an invariant $\phi \in PSH(X,\omega)$ such that
\[
 g(m_\phi)MA(\phi) = \mu,
\]
where $m_\phi$ is a moment map for the torus action induced by $\phi$.
\end{ThmA}

As already suggested in~\cite{BermanWitt} the proof follows from the result of McCann~\cite{McCann}, although not in a straightforward way. The only gap to fill from there is to ensure that appropriate notions of convergence for real and complex solution coincide. Towards this end we prove the following lemma:

\begin{LemA}
 If a unifiormly Lipschitz sequence of convex functions $F_n$ converges in a monotone way to a convex function $F$, then their Legendre transforms $F^*_n$ converge to $F^*$ in $W^{1,\infty}_{loc}$.
\end{LemA}

As already mentioned, in the real setting this is a well studied equation that appears in the theory of optimal transportation. In the complex case, it comes up as an equation for K\aumlaut hler-Ricci solitons on Fano varieties, although in that case the action might not be completely integrable.

\paragraph{Acknowledgement} The author would like to thank Sławomir Dinew for his guidance. The author was supported by Polish National Science Centre grant 2018/29/N/ST1/02817.

\section{Background material}

\subsection{Convex functions}




 Here we want to recall a few facts about convex functions. For a convex function $u:\rzeczywiste^n \rightarrow \rzeczywiste \cup \{+\infty\}$ we define its domain as the convex set $\{x\;:\;u(x) < \infty\}$ and denote it as $\dom(u)$. For convenience we exclude the function $u \equiv +\infty$ from the set of convex functions.
 
For any convex function $u$ on $\rzeczywiste^n$ its Legendre transform is defined by 
\[
 u^*(p):= \sup_{x\in\rzeczywiste^n}\{\langle x,p\rangle - u(x)\}.
\]
It is a crucial notion in convex analysis. It is not hard to show that the Legendre transform $u^*$ is a convex lower semicontinuous function.

The multivalued subgradient of $u$ is a set-valued map defined on $\text{int}(\dom(u)$) that attaches to a point the set of slopes of supporting planes at that point, namely
\[
 y\in\partial u(x) \Leftrightarrow \; \forall z \in \rzeczywiste^n \;\;\; u(x) + \langle y, z-x \rangle \leq u(z).
\]
Since $u$ is convex $\partial u$ is always non-empty on $\text{int}(\dom(u))$. It is single-vauled iff $u$ is differentiable at $x$ and at this point it is equal to $\nabla u(x)$.

The notion of a subgradient is closely related to the notion of a Legendre trasform through the following equivalences
  \[
   x\cdot p = u(x) + u^*(p) \Leftrightarrow p \in \partial u(x) \Leftrightarrow x \in \partial u^*(p).
  \]
From this, one can see that in the case of a smooth strictly convex function the gradient of the function and of its Legendre transform are each other's bijective inverses.

The most important fact concerning the differentiability of convex functions is the following one:
\begin{Thm}[Rademacher's Theorem]\label{Lem:Rademacher}
 Any convex function is differentiable on a subset of full measure of its domain.
\end{Thm}

The set of differentiability of $u$ will be denoted by $\dom(\nabla u)$.

Finally, we list the properties of convex subgradients that will be of use to us. 

\begin{Prop}[Closedness of subgradients, {\cite[Theorem 24.4]{Rockafellar}}]\label{Prop:subgradient-domkniety}
 For any convex function $u$ the graph of $\partial u$ is a closed subset of $\rzeczywiste^n \times \rzeczywiste^n$.
\end{Prop}

\begin{Prop}[Boundedness of subgradients, {\cite[Corollary 24.5.1]{Rockafellar}}]
 Pick any $x \in \text{int}\,(\dom(\phi))$, then for any positive $\epsilon$ there is a positive $\delta$ such that
 \[
  \partial\phi(B(x,\delta)) \subseteq \partial\phi(x) + B(0,\epsilon),
 \]
 with $B(x,\delta)$ denoting the ball centered at $x$, with the radius $\delta$.
\end{Prop}
This gives us the following corollary:
\begin{Cor}\label{Cor:ogr-na-zwartych}
 For any convex function $u$ and for any compact subset $K$ of $int\,(\dom(u))$ the set $\partial u(K)$ is bounded.
\end{Cor}
\begin{proof}
 Indeed, pick an $\epsilon$ and take appropriate $\delta_x$-balls at points $x$ in $K$. This gives a covering of $K$. The only thing left to show is that at $x \notin \dom(\nabla u)$ the subgradient $\partial u(x)$ is still bounded. But $x$ is in $\text{int}(\dom(u))$, if the subgradient would be unbounded then $u$ would get arbitrarily big arbitrarily close to $x$, which cannot be since $x$ lies a positive distance away from the boundary of $\dom(u)$.
\end{proof}

\begin{Lem}\label{Lem:rownosc-pw}
 Let $u,v$ be two convex functions defined on some convex set $C$ with non-empty interior. If $\{\nabla u = \nabla v\}$ is a subset of full measure of $C$ then $u\equiv v$ in $\text{int}(C)$ modulo additive constant.
\end{Lem}

\begin{proof}
 If $\rho_{\epsilon}$ is the standard mollifier then it is easy to verify that $u\ast \rho_{\epsilon}$ and $v\ast \rho_{\epsilon}$ are convex in $C^\epsilon$ smooth and convegre locally uniformly to $u$ and $v$ respectively. Where the set $C^\epsilon$ is the domain of definition of mollified function, i.e. $\{x\in C \;|\; \text{dist}(x,\partial C) > \epsilon\}$.  By smoothness, the almost everywhere equality of their gradients implies their equality everywhere and thus $u\ast \rho_{\epsilon}$ and $v\ast \rho_{\epsilon}$ must converge to the same function up to a constant.
\end{proof}

\subsubsection{Convergence of convex functions}
The following facts about the convergence of convex functions will be of use. The most natural notion is the following.

\begin{Lem}
 If a sequence of convex functions $\{u_k\}$ converges locally uniformly to a function $u$, then $u$ is convex.  
\end{Lem}

We would also like to say something about the convergence of subgradients.

\begin{Def}
 We say that the sequence of subgradients $\partial u_n$ converges graphically to subgradient $\partial u$ if their graphs converge as sets, i.e.
 \begin{align*}
  \text{graph}(\partial u) = \{(x,p) \; |\; \exists \; (x_n,p_n) \in \rzeczywiste^n\times \rzeczywiste^n \,: p_n \in \partial u_n(x_n) \\
  \text{and}\, (x_n,p_n) \rightarrow (x,p) \}.
 \end{align*} 
\end{Def}

The following theorem of Attouch is a fundamental result concerning the graphical convergence of subgradients.
\begin{Thm}[{\cite[Theorem 12.35]{RockafellarWets}}]\label{Thm:Attouch}
 For convex, lower semicontinuous functions $f_n$ and $f$ the following are equivalent:
 \vspace{-1em}
 \begin{enumerate}
  \item $f_n \rightarrow f$ locally uniformly,
  \item $\partial f_n \rightarrow \partial f$ graphically and for some choice of $p_n \in \partial f_n(x_n)$ and $p \in \partial f(x)$ such that $(x_n,p_n) \rightarrow (x,p)$ one has $f_n(x_n) \rightarrow f(x)$.
 \end{enumerate}
\end{Thm}

\begin{Rem}
 The first part of the orginal theorem is expressed in terms of ``epigraphical'' convergence, but is equivalent to locally uniform convergence (see~\cite[Theorem 7.17]{RockafellarWets}).
\end{Rem}

Finally, we would like to describe the relationship between the graphical convergence and pointwise convergence, for this we need the following definition:

\begin{Def}
 The sequence of set-valued maps $S_n$ is equicontinuous at point $x$ with respect to subset $X$ if for each positive $\epsilon$ there is a neighbourhood $V$ of $x$ such that 
for almost every $n$
\[
 S_n(y) \subset S_n(x) + B(0,\epsilon) \;\;\text{for all}\; y \in V \cap X.
\]
The sequence is equicontinuous with respect to $X$ if it is equicontinuous at each point of $X$.
\end{Def}

\begin{Rem}
 In~\cite{RockafellarWets} the above  notion is called the asymptotic equi-outer-semi\-con\-tinu\-ity. Since we don't need other notions of 
equicontinuity we will just call that one the equicontinuity.
\end{Rem}

Now the deisired relationship is the following:

\begin{Thm}[{\cite[Theorem 5.40]{RockafellarWets}}]\label{Thm:rowno-grafo}
 For the sequnece of set-valued maps $S_n$, the map $S$ and a set $X$ any pair of the the following conditions implies the third:
 \begin{enumerate}
  \item $S_n$ is equicontinuous with respect to $X$,
  \item $S_n$ converges graphically to $S$ relative to $X$,
  \item $S_n$ converges pointwise to $S$ relative to $X$.
 \end{enumerate}
\end{Thm}

\subsubsection{The class of globally Lipschitz convex functions}

\begin{Def}
 By support function of a bounded convex set $P$ contatining zero we mean the function
 \[
  \phi_P(x) := \sup_{p\in P}\langle x,p\rangle.
 \]
 If the set $P$ is bounded then $\phi_P$ is finite everywhere.
\end{Def}

\begin{Def}
 We will denote by $\mathcal{P}$ the space of convex functions dominated by $\phi_P$, i.e. the set $\{u - \text{convex}\;|\; \exists\,C: \; u \leq \phi_P + C \}$. This is the set of convex functions whose Legendre transform is $+\infty$ outisde $P$. 
 The subset of $\mathcal{P}$ consisting of functions that also dominate $\phi_P$ will be denoted by $\mathcal{P}_{+}$, in other words $\mathcal{P}_+ = \{u \in \mathcal{P}\; |\; \exists\, C:\; u-C \geq \phi_P\}$. Both sets can be equipped with the topology of pointwise convergence, which is equivalent to the topology of locally uniform convergence, by the virtue of uniform Lipschitz constant for all $\mathcal{P}$.
\end{Def}

The set $\mathcal{P}_+$ is dense in $\mathcal{P}$. Moreover the approximating sequence can be chosen as nice as possible.

\begin{Lem}[{\cite[Lemma 2.2]{BermanBernd}}]\label{Lem:aproksymacja}
 Every $\phi \in \mathcal{P}$ can be approximated by decreasing seqence of smooth strictly convex functions from $\mathcal{P}_+$.
\end{Lem}

\subsection{Optimal transport and Monge-Amp\`{e}re equation}

We say that the function $T:\rzeczywiste^n \rightarrow \rzeczywiste^n$ transports probability measure $\mu$ to probability measure $\nu$ if for any Borel set $A$ 
the following equality holds
\[
 \nu[A] = \mu[T^{-1}(A)].
\]
Alternatively we say that $T$ pushes $\mu$ forward to $\nu$ and denote the push-froward measure by $T_{\#}\mu$.

In general there will be a lot of such maps, so it is natural to put some optimality constraints on them. The best understood contraint and in some cases the natural one is minimizing the quadratic cost, i.e. the transport map should  minimize the following functional
\[
 \int_{\rzeczywiste^d} |x - T(x)|^2 \,d\mu.
\]
In general there might not be a solution and if it exists it might not be unique, some regularity assumptions for the measures must be added. For example, one can assume that the measures have finite second moments and $\mu$ is absolutely continuous. In that case the solution exists and has a form of $T=\nabla \phi$ for some convex function $\phi$. For thorough discussion of this problem, the reader might consult~\cite{Villani}.

Supposing that a solution exists, by the trasport condition we get
\[
 \int\chi_A \,d\nu = \int_A\, d\nu = \int_{(\nabla\phi)^{-1}(A)}\,d\mu = \int \chi_A\circ \nabla\phi\, d\mu
\]
That can easily be generalized to get that for any $f \in C_b(\rzeczywiste^n)$
\begin{equation}\label{eq:rownanie}
 \int f \,d\nu = \int f\circ\nabla\phi\, d\mu. 
\end{equation}
Here $C_b(\rzeczywiste^n)$ denotes the set of continuous and bounded functions on $\rzeczywiste^n$.

Suppose now that $d\nu = g(x)dx$ for some density $g(x)$ and $\phi$ is a $C^2$ function. By change of variables formula we get that
\[
  \int f(\nabla\phi(x))\, d\mu = \int f(\nabla \phi(x))g(\nabla \phi(x))\det D^2\phi\, dx
\]
and that provides one with a notion of solution to the transported \Monge equation
\[
 MA^{\rzeczywiste}_g(\phi) := g(\nabla \phi(x))\det D^2\phi = \mu
\]
as long as the optimal transport map exists.

As we mentioned, for any two probability measures the optimal transport solution might not exist. However, under a mild regularity assumption it is still possible to transport one to another through a subgradient of convex function, so that the condition~(\ref{eq:rownanie}) is still satisfied. This is the content of the following important theorem. 

\begin{Thm}[McCann~\cite{McCann}]
 Let $\mu, \nu$ be probability measures on $\rzeczywiste^n$ and suppose that $\mu$ vanishes on Borel subsets of $\rzeczywiste^n$ of Hausdorff diemnsion $n-1$. Then there exists a convex function $\psi$ on $\rzeczywiste^n$ whose subgradient $\partial\psi$ pushes $\mu$ forward to $\nu$. $\partial\psi$ is uniquely deterined $\mu$-almost everywhere.
\end{Thm}

Of course the assumption on the null sets of $\mu$ can not be abandoned. For example if $\mu = \delta_x$ and $\nu$ is not a point measure, then if $A$ is such a set that $0 < \nu[A] < 1$ one gets that for any convex function $\phi$, $\nu[A] \neq \mu[(\partial\phi)^{-1}(A)]$ since the latter must always be either 0 or 1.

\subsection{Torus action}

As in the intrduction we are interested in completely integrable K\aumlaut hler mani\-folds. In this setting the following results provide the correspondence between the K\aumlaut hler geometry and convex functions.

\begin{Prop}[\cite{Guillemin}]
 There is an open dense subset $X_0 \subset X$ where the action of $T_c^n$ is free, making $X_0$ diffeomorphic to $(\zespolone^*)^n$. Every invariant K\aumlaut hler form $\omega$ on $X$ has a K\aumlaut hler potential on $X_0$, i.e.
 \[
  \omega|_{X_0} = 2i\partial\bar{\partial}F
 \]
 for some $F$.
\end{Prop}
The set $X\setminus X_0$ is given as a vanishing set of some holomorphic vector fields, so it must be analytic.

If we introduce coordinates on $X_0$ coming from $(\zespolone^*)^n$ by $L: e^{x+ iy} \rightarrow x +iy$, the invariance of potential means that the function $F$ from the previous proposition depends only on $x$ variable in $\rzeczywiste^n$ and positive definiteness means that $F$ must be convex. Moreover, nothing in the proof actually requires the form to be smooth, so the conclusion easily extends to forms with more singular  coefficients, thus asserting that every closed positive and invarinat $(1,1)$-current in the cohomolgy class $[\omega]$ will admits a convex potential.

\begin{Prop}\label{Prop:moment}
 For the symplectic form $\omega$ as above, the moment map is
 \[
  \frac{\partial F}{\partial x} + c,
 \]
 with $c$ being any constant vector in $\rzeczywiste^n$.
\end{Prop}
 Finally, we recall the theorem of Atiyah~\cite{Atiyah}, Guillemin and Sternberg~\cite{GuilleminSternberg}:
\begin{Thm}
 The image of $X$ through the moment map is a compact convex polytope in $\rzeczywiste^n$.
\end{Thm}

In the case of completely integrable actions the polytopes that can arise as images of moment maps are called Delzant polytopes. Conversely for each Delzant polytope there exists a K\aumlaut hler manifold with completely integrable torus action and a moment map that maps to this polytope.

\subsection{Toric pluripotential theory}

The class of plurisubharmonic functions that are torus invariant will be denoted by $PSH_{tor}(X,\omega) = \{\phi \in PSH(X,\omega)\; |\; \forall z\in X, t\in T^n \;| \; \phi(t\cdot z) = \phi(z) \}$. The results of the preovious section imply that to each such function corresponds a convex function on $\rzeczywiste^n$.

More precisely, if set $X_0$ are coordinate map $L$ are as in the previous subsection then for any $v \in PSH_{tor}(X,\omega)$ the form $\omega_v = \omega + i\partial\bar\partial v$ is  still invariant and closed there, so restricting to $X_0$ there is a convex $F_v$ function given by
\[
 F_v \circ L = F_0 \circ L + v, 
\]
with $F_0\circ L$ being the potential for $\omega$. Of course if use the formula above to produce a plurisubharmonic function it will only be defined on $X_0$, but since $X\setminus X_0$ is analytic, the function will extend to the whole $X$.

Not every convex function can be a potential for an invariant K\aumlaut hler form. If $P$ is the Delzant polytope of the manifold $(X,\omega)$ then the following Propostion holds (see e.g.~\cite{CGSZ} for a proof).

\begin{Prop}
  The following are equivalent:
  \begin{enumerate}
    \item $v \in PSH_{tor}(X)$,
    \item $F_v \in \mathcal{P}$.
  \end{enumerate}
\end{Prop}

Finally, the two \Monge measures coincide up to a constant. Specifically, the complex \Monge measure is defined as
\[
 MA^{\zespolone}(v)[A] := \int_A (\omega + i\partial\bar\partial v)^n
\]
and the real measure as
\[
 MA^{\rzeczywiste}(F)[B] := |\partial F(B)|
\]
for Borel sets $A,B$ in $\zespolone^n$ and $\rzeczywiste^n$ respectively. Here $|\cdot|$ is the Lebesgue measure. For $C^2$ convex functions it coincides with the measure
\[
 MA^{\rzeczywiste}(F)[B] := \int_B \det(D^2 F).
\]

For invariant plurisubharmonic funtion the two concepts are connected through the following proposition.

\begin{Prop}\label{Prop:rownosc-miar}
 Let $\phi \in PSH_{tor}(X)$, we identify $X_0$ with $(\zespolone^*)^n$. Then for any $f \in C_b(\rzeczywiste^n)$
 \[
  \int_{X_0} (f\circ L) \, MA^{\zespolone}(\phi) = \frac{n!}{(2\pi)^n}\int_{\rzeczywiste^n} f \,MA^{\rzeczywiste}(F_\phi).
 \]
\end{Prop}

The proof is just a straightforward computation (see e.g.~\cite{CGSZ}) in the smooth case and then the application of classical convergence theorems for convex and plurisubharmonic functions.

\subsubsection{g-\Monge measure}

Following the preprint~\cite{BermanWitt} we define the complex $g$-\Monge measure or the complex transported \Monge measure as
\[
 MA_g(\phi) = g(\mu_\phi)MA(\phi).
\]

From the Proposition~\ref{Prop:moment} and the definition of the real transported \Monge measure it is not hard to see that~\ref{Prop:rownosc-miar} extends for smooth functions to transported measures. In the more general case, especially with the torus of smaller rank, the definition becomes more intricate. 

If we denote by $\mathcal{E}_g$ the set of all $PSH_{tor}$ functions with full $MA_g$ mass, i.e. those functions for which
\[
 \int_X MA_g(\phi) = \int_P g(p)dp = 1
\]
then the following crucial continuity statement holds:
\begin{Thm}[\cite{BermanWitt}, Theorem 2.7]\label{Thm:zbieznosc-zesp}
 If $\phi_j$ is a sequence in $\mathcal{E}_g$ decreasing to $\phi$ in $\mathcal{E}_g$ then
 \[
  MA_g(\phi_j) \rightarrow MA_g(\phi)
 \]
 in the weak topology of measures.
\end{Thm}

\section{Full rank existence and uniqueness}

Given a probability measure $g(p)dp$ on $P$ and \textit{any} probability measure $\mu$ on $\rzeczywiste^n$, we would like to solve the equation
\[
 MA^{\rzeczywiste}_g(u) = \mu
\]
in some appropriate sense. One can not apply McCann's theorem directly since for example $\mu = \delta_x$ would prevent the existence of the transport map, thus we must use the regularity of $g$.

Suppose that we have a smooth strictly convex solution $u$, so that every term in $MA_g^{\rzeczywiste}(u)$ is well-defined and moreover so is $\nabla u^*$. By the fact that for any $x$ and any $p$, $\nabla u(\nabla u^*(p)) = p$ and $\nabla u^*(\nabla u(x)) = x$ we define the solution through the change of variables formula. Thus
\begin{equation}\label{eq:definicja}
 \int_{\rzeczywiste^n} f(x)g(\nabla u(x))MA^{\rzeczywiste}(u) = \int_{P}f(\nabla u^*(p))g(p)dp = \int_{\rzeczywiste^n} f\,d\mu
\end{equation}
and a function $u \in \mathcal{P}$ such that the second equality holds for any contiunous bounded function $f$ is defined to be a solution.

The fact that there is such a solution follows easily from McCann's theorem. Suppose that $\phi$ is the convex function whose gradient transports $g(p)dp$ (understood as a measure on $\rzeczywiste^n$) to $\mu$. By the regularity of $g$ and McCann's theorem it must exist. Then $\nabla\phi$ is defined $g(p)dp$-almost everywhere and since $P$ is convex we can take $\phi$ to be $+\infty$ outside of $P$. Thus after possibly fixing $\phi$ on $\partial P$ so that it is lower semi-continuous, its Legendre transform $\phi^*$ becomes unique and defined everywhere on $\rzeczywiste^n$ and thus belongs to the class $\mathcal{P}$ since by lower semicontinuity $\phi^{**} = \phi$. The convex function $u=\phi^*$ is the unique (up to additive constant) solution to the transported \Monge problem in the class $\mathcal{P}$. Indeed, since $\nabla\phi$ transports $g(p)dp$ to $\mu$ it means that for any $f\in C_b(\rzeczywiste^n)$
\[
 \int_{\rzeczywiste^n} f\,d\mu = \int_P f(\nabla\phi(p))g(p)dp = \int_P f(\nabla u^*(p))g(p)dp.
\]
If there was to be another solution $v$ in the class $\mathcal{P}$ then its Legendre transform would have been $+\infty$ on the complement of $P$ and lower semicontinuous on its boundary and it would induce a transport of $g(p)dp$ to $\mu$, so by McCann's uniqueness theorem $\nabla u^* = \nabla v^*$ $g\,dp$-almost everywhere, and since $g > 0$, by Lemma~\ref{Lem:rownosc-pw} we get $u^* = v^*$ (mod $\rzeczywiste$) everywhere on $\text{int}(P)$.

\vspace{-1em}
\paragraph{Remark on uniqueness.} Of course the uniqueness statement becomes false if we allow functions outside of class $\mathcal{P}$. Suppose that $\mu = \delta_0$, then the solution is obviously $u=\phi_P$, so that $u^*\equiv 0$ on $P$. But now adding to $u$ any convex function $v$ such that $\min v = v(0)$ would also give a solution, since $(u+v)^*\equiv 0$ on $P$.

\begin{Ex}
 Finally, we would like to point out that if the assumptions of McCann's theorem are not satisfied for at least one of the measures there might not be a weak solution. For example, if $\mu =\delta_{[-1,1]\times \{0\}}$ and $\nu = \delta_{\{0\}\times [-1,1] }$ are measures on $\rzeczywiste^2$, then it is impossible to find a solution. Indeed, the only candidate is $u(x,y) = |y|$ and it is easy to see that it can not be the solution since $u^*|_{\{0\}\times [-1,1]} \equiv 0$.
\end{Ex}

\subsection{The complex case}

\begin{Cor}
 The solution to the real problem in $\rzeczywiste^n$ induces a unique solution to the $g$-\Monge problem on toric manifolds.
\end{Cor}

\begin{proof}
 Firstly, we notice that the fact that $\mu$ does not put any mass on pluripolar sets implies that $X\setminus X_0$ as an analytic set has no mass. Thus we can restrict the problem to $X_0$. Moreover, since the measure is invariant, it can be interpreted as a measure on $\rzeczywiste^n$ also denoted by $\mu$. 
 
 Suppose now we have a real solution $F_\phi$ for the measure $\mu$, then one suspects that $\phi = (F_\phi - F_0)\circ L $ would be the solution for the  corresponding invariant measure. Indeed, $F_\phi$ is in $\mathcal{P}$, so it must correspond to some invariant psh function. Moreover, for smooth strictly convex functions the formula~(\ref{eq:definicja}) obiously translates by Proposition~\ref{Prop:rownosc-miar} to the complex setting. Finally, by Lemma~\ref{Lem:aproksymacja} there exists a decreasing sequence $F_n$ of smooth strictly convex functions that decreases to $F_\phi$, so by smoothness and Theorem~\ref{Thm:zbieznosc-zesp} $MA^{\rzeczywiste}_g(F_n) = MA^{\zespolone}_g(F_n-F_0)$ converges weakly to $MA^{\zespolone}_g(F_\phi)$. Thus the only thing left to show is that $MA^{\rzeczywiste}_g(F_n)$ converges weakly to $\mu$. 

Take $f \in C_b(\rzeczywiste^n)$ and put $f_n:= f\circ\nabla F^*_n$. We would like to show that $f_n$ converge almost everywhere to $f$. That would give us the desired assertion by the dominated convergence. 

First, let us prove that decreasing convergence of $F_n$ implies locally uniform convergence of $F_n^*$. Indeed, since $F_n$'s are uniformly Lipschitz, their pointwise convergence implies locally uniform convergence. Now $F \leq F_n$ implies $F^*_n \leq F^*$, take $p \in \text{int}P$ and suppose that the supremum in $F^*(p)$ is realized by $x^*$, thus
\[
 F^*(p)-F^*_n(p) = \sup_{x\in\rzeczywiste^n}\{\langle x,p\rangle - F(x)\} + \inf_{x\in\rzeczywiste^n}\{F_n(x)-\langle x,p\rangle\} \leq F_n(x^*) - F(x^*).
\]
Thus $F_n^*$ converges pointwise to $F^*$. If $K \subseteq \text{int}(P)$ is compact then by Proposition~\ref{Prop:subgradient-domkniety} and Corollary~\ref{Cor:ogr-na-zwartych} $\partial F^*(K)$ is compact and for every $q \in K$ the supremum in $F^*(q)$ is realized by some $y^*$ in $\partial F^*(q)$, thus the convergence is locally uniform.

Now, we will show that $F^*_n$ converging locally uniformly to $F^*$ implies that $\nabla F_n^*$ converges to $\nabla F^*$ almost everywhere and that would finish the proof. To do that we want to employ the Theorems~\ref{Thm:Attouch} and~\ref{Thm:rowno-grafo} restricted to $\dom(\nabla F^*)$. Thus the only thing left to show is the equicontinuity of $\nabla F^*_n$'s with respect to $\dom(\nabla F^*)$.  

In order to prove this we will first prove the following lemma:
\begin{Lem}
 $\nabla F^*_n$ are locally bounded independently of $n$.
\end{Lem}

\begin{proof}
Take any point $x \in \text{int}P \cap \dom(\nabla F^*)$ and pick a positive $\delta$ such that $B(x,\delta)$ is relatively compact in $\text{int}P$. By the boundedness of the subgradient (Corollary~\ref{Cor:ogr-na-zwartych}) there exists a positive $M$ such that
\[
 \partial F^*(B(x,\delta)) \subseteq B(\nabla F^*(x),M).
\]
  Now pick a positive $\eta$, starting from some $n$ we get that $0 \leq F^* - F^*_n \leq \eta$ over $B(x,\delta)$. We claim that $\nabla F_n^*(B(x,\delta/2)) \subset B(\nabla F^*(x),M + C)$ holds for some constant $C$, independent of $F^*_n$. 
  
  Indeed, by convexity it is enough to estimate the gradients on the boundary of $B(x,\delta/2)$. Take a point $y\in \partial B(x,\delta/2)$ such that $|\nabla F^*_n(y)|$ achieves maximum over $\partial B(x,\delta/2)$. The vector $\nabla F^*_n(y)$ must be pointed to the outside of $B(x,\delta/2)$ or at least be tangent to it. The ``boundary'' steepset case is $F^*_n(y) = F^*(y) -\eta$, $F^*$ growing at best possible rate from $y$ and tangent plane at $F^*_n(y)$ touching $F^*$ at the boundary of $B(x,\delta)$, then $\nabla F^*_n(y)$ would become the steepest if that happened over shortest possible interval which would be of length $\delta/2$. Thus finally $|\nabla F^*_n(y)| \leq \eta + p\delta/2$, where $p$ is the length of the longest vector in $B(\nabla F^*(x),M)$.
\end{proof}

With the Lemma in hand the rest of the proof is straightforward. Suppose the sequence is not equicontinuous at some point $x_0 \in \text{int}P \cap \dom(\nabla F^*)$. Thus there is a positive $\epsilon$ such that for any $k$ there is $y_k \in B(x_0,1/k) \cap \dom(\nabla F^*)$ such that 
\[
 |\nabla F_{n(k)}(x_0) - \nabla F_{n(k)}(y_k)| > \epsilon
\]
 with $n(k)$ being some subsequence of $\mathbb{N}$. But by above lemma the set $p_k=\{\nabla F^*_{n(k)}(y_k)\}$ is bounded, thus there must be a convergent subsequence, conviniently also named  $p_k$, such that $p_k \xrightarrow{k\rightarrow\infty} p$. But the set $q_k=\{\nabla F^*_{n(k)}(x_0)\}$ is also bounded thus a subsequence must converge to some $q$ such that $|q-p|\geq \epsilon$. Thus we have two subsequences $(y_k,p_k)$ and $(x_0.q_k)$. By graphical convergence both of them must converge to some point in $\partial F^*(x_0)$, but this set is a singleton and that is a contradiction.
\end{proof}






 


\vspace{1em}
\textsc{Szymon Myga, Department of Mathematics and Computer Science, Jagiellonian University, Poland}

\textit{E-mail address: }\texttt{\email}

\end{document}